\documentclass[a4paper, twoside, 10pt]{article}
\usepackage{eurosym, xcolor}
\usepackage{amsmath,amsthm,amsfonts,latexsym,amscd,amssymb, enumerate}
\usepackage{hyperref}
\numberwithin{equation}{section}
\input xypic
\newtheorem{theorem}{Theorem}[section]
\newtheorem{lemma}{Lemma}[section]

\newtheorem{proposition}{Proposition}[section]
\newtheorem{definition}{Definition}[section]

\theoremstyle{remark}

\date{}
\title{\textbf{On the geometry of Clairaut warped product Riemannian maps}}
\author{Jyoti Yadav and Gauree Shanker\thanks{corresponding author, Email: gauree.shanker@cup.edu.in}}

\begin{document}
	\maketitle
	\begin{abstract}
In this paper, we introduce Clairaut warped product Riemannian maps. To study these kind of maps, first, we find the condition of geodesic of a regular curve. Then we obtain the conditions for a warped product Riemannian map to be Clairaut warped product Riemannian map. Further, we find the curvature tensor for this map.		
	\end{abstract}
	
	\section{Introduction}
	As a generlizations of product manifold, warped product manifold was defined by Bishop and O'Neill \cite{J16}. With the help of warped product, one can construct new examples of manifolds of negative curvature by defining convex function on this. Schwarzschild solution
	and Robertson-Walker models are the well known examples of warped products. In Riemannian geometry, de Rham theorem tells that Riemannian manifold can be locally decomposed into product manifold. After this, Moore \cite{J17} demonstrated the sufficient conditions for an isometric immersions into a Euclidean space to decompose into a product imersions. 
	The definition of warped product immersion is as follows. 
	Let $\phi_i: M_i\rightarrow N_i$ be isometric immersions from $M_i$ to $N_i$ for $i= 1,...,k.$ Let $M_1\times_{f_1} M_2\times_{f_2},...,\times_{f_k}M_k$ and $N_1\times_{\rho_1} N_2\times_{\rho_2},...,\times \rho_kN_k$ be warped product manifolds, where $\rho_i:N_i\rightarrow\mathbb{R}^+$ and $ f_i:\rho_i\circ\phi_i:M_i\rightarrow\mathbb{R}^+$ for $i=1,...,k-1.$ Then the smooth map $\phi:M_1\times_{f_1} M_2\times_{f_2},...,\times f_kM_k\rightarrow N_1\times_{\rho_1} N_2\times_{\rho_2},...,\times \rho_kN_k$ defined by $\phi(p_1,...,p_k)= (\phi_1(p_1),...,\phi_k(p_k))$ is an isometric immersion. For futher study of warped product isometric immersion, we refer \cite{J9}.\\
	In Riemannian geometry, to formulate Riemannian manifolds with positive or non-negative sectional curvature is an important problem. By using Riemannian submersion, many examples of Einstein manifold have been constructed. The idea of Riemannian submersion was intoduced by Gray and O'Neill \cite{J10}. For further study of Riemannian submersion, we refer the book \cite{J13} .\\
	Recently, geometry of Riemannian warped product maps has been introduced by Meena et.al. In this paper, we introduce combination of two maps which is Clairaut warped product Riemannian maps.\\ 
	The paper is organized in the following manner. In section 2, we recall the results which are required for current research work. In section 3, firstly, we find the conditions for a regular curve to be a geodesic. Using these conditions, we find the necessary and sufficient condition for a warped product Riemannian map to be a Clairaut warped product Riemannian map. Then, we calculate the Ricci curvature tensor for Clairaut warped product Riemannian maps.

%	 According to Nash embedding theorem, every Riemannian manifold can be isometrically immersed in some euclidean space. With the help of this theorem, One can easily say that every warped product manifold can be isometrically immersed into some euclidean space.
	\section{Preliminaries}
Here, we recollect some definitions and results which are the base of this paper.
	\subsection{Riemannian map and O'Neill tensor}
	\begin{definition}
		A smooth map $F$ is said to be a Riemannian map at $p\in M$ if the horizontal restriction
	$F^h_{p}:(kerF_{*p})^\perp\rightarrow (rangeF_{*p})$ is a linear isometry between the inner product spaces $((kerF_{*p})^\perp, g_M{(p)}| (kerF_{*p})^\perp)$ and $((rangeF_{*p}, g_N(q)|(rangeF_{*p})), q = F(p)$.
	\end{definition}

	O'Neill \cite{J10} defined fundamental tensor fields $T$
	and $A$ by
	\begin{equation}\label{neill A}
		A_{X'}{Y'} = \mathcal{H}\nabla^M_{\mathcal{H}X'}\mathcal{V}Y' + \mathcal{V}\nabla^M_{\mathcal{H}X'}\mathcal{H}Y',\\
	\end{equation}
	\begin{equation}\label{neill T}
		T_{X'}{Y'} =  \mathcal{H}\nabla^M_{\mathcal{V}X'}\mathcal{V}Y' +\mathcal{V}\nabla^M_{\mathcal{V}X'}\mathcal{H}Y',\\
	\end{equation}
	where $X', Y'\in\Gamma(TM), \nabla^M$ is the Levi-Civita connection of $g_M$ and $\mathcal{V}, \mathcal{H}$ denote the projections to vertical subbundle and horizontal subbundle,
	respectively. For any $X'\in \Gamma(TM)$, $T_{X'}$ and $A_{X'}$ are skewsymmetric
	operators on $(\Gamma(TM), g_M)$ reversing the horizontal and the vertical distributions.
	One can check easily that $T$ is vertical, i.e., $T_{X'} = T_{\mathcal{V} X'}$, and $A$ is horizontal, i.e., $A = A_{\mathcal{H}X'}$.\\
	We observe that the tensor field $T$ is symmetric for vertical vector fields, %i.e.,
	%\begin{equation*}
	%	T_{V}W = T_{W}V ~~\forall~~ V,W\in\Gamma(kerF_*),\\ 
	%\end{equation*}
	and tensor field $A$ is anti-symmetric for horizontal vector fields. 
	%\begin{equation*}
	%	A_{X}Y = -A_{Y}X = \frac{1}{2}\mathcal{V}[X, Y] ~~\forall~~X,Y\in \Gamma(kerF_{*})^\perp.\\
	%\end{equation*}
	Using  equation \eqref{neill A} and \eqref{neill T}, we have the following Lemma.\\
	\begin{lemma} \cite{J10}\label{Neill tensor} Let  $X, Y\in\Gamma (kerF_{*})^\perp$ and $V,W \in \Gamma(kerF_{*}).$ Then
		\begin{equation*}
			\nabla_{V}W = T_{V}W + \hat{\nabla}_{V}W,
		\end{equation*}
		\begin{equation*}
			\nabla_{V}X = \mathcal{H}\nabla_{V}X + T_{V}X,
		\end{equation*}
		\begin{equation*}
			\nabla_{X}V = A_{X}V +\mathcal{V}\nabla_{X}V,
		\end{equation*}
		\begin{equation*}
			\nabla_{X}Y = \mathcal{H}\nabla_{X}Y + A_{X}Y.
		\end{equation*}
		If $X$ is a basic vector field. then $\mathcal{H}\nabla_{V}X = A_X V.$
	\end{lemma}
%	For any vector field $X'$ on $M$ and any section $D$ of $(rangeF_*)^\perp$, we define the connection $\nabla^{F\perp}$ as the orthogonal projection of $\nabla_{X'} D$ on $(rangeF_*)^\perp.$ One can easily see that $\nabla^{F\perp}$ is a linear connection on $(rangeF_*)^\perp$ with $\nabla^{F\perp}g_N = 0.$
%	For a Riemannian map $F$, we have  \cite{S23}
%	\begin{equation}\label{S_V}
%		\nabla^{N}_{F_{*}X}D = -S_D{F_{*}X} + \nabla^{F \perp}_ {X} D,
%	\end{equation}
%	where $S_D{F_{*}X}$ is the tangential component and  $\nabla^{F \perp}_ {X} D$ is the orthogonal component  of $\nabla^{N}_{F_{*}X}D$.
%	It can be easily seen that $\nabla^{N}_{F_{*}X}D$ is obtained from the pullback connection of $\nabla^{N}$. Thus, at $p\in M$, we have 
%	$\nabla^{N}_{F_{*}X}D(p) \in T_{F(p)}N, S_D{F_{*}X}\in F_{*p}(T_{p}M)$ and $\nabla^{F \perp}_ {X} D\in (F_{*p}(T_{p}M))^\perp$. It follows that $S_D{F_{*}X}$ is bilinear in $D$,  $F_{*}X$ and $S_D{F_{*}X}$  at $p$ depends only on $D_{p}$ and $F_{*p}X_{p}$.\\  By direct computations, we obtain 
%	\begin{equation*}\label{Sv}
%		g_N(S_D{F_{*}X, {F_{*}Y}}) = g_N \big(D, (\nabla F_{*})(X, Y)\big)\\ 
%	\end{equation*} 
%	for all X, Y $\in \Gamma(kerF_{*})^\perp~and ~V\in\Gamma(rangeF_{*})^\perp.$
%	Since $\nabla F_{*}$ is symmetric, it follows 	that $S_D$ is a symmetric linear transformation of $range F_{*}$.  \\\
	\subsection{Riemannian warped product manifold}
	Let $(M_1, g_{M_1})$ and $(M_2, g_{M_2})$ be two finite dimensional Riemannian manifolds with $dimM_1 = m_1, dimM_2 = m_2$
	respectively. Let $f$ be a positive smooth function on $M_1.$ The warped product
	$M = M_1 \times_f M_2$ of $M_1$ and $M_2$ is the Cartesian product $M_1 \times M_2$ with the
	metric
	$g_M = g_{M_1} + f^2g_{M_2}.$ In addition, Riemannian metric $g_M$ on $M = (M_1\times_fM_2)$ is defined for any two vecor field in the following manner.
	$$g_M(X, Y) = g_{M_1}(\pi_{1*}X, \pi_{1*}Y) + f^2(\pi_1(p_1)) g_{M_2}(\pi_{2*}X, \pi_{2*}Y)$$ where $\pi_1 : M_1 \times_f M_2 \rightarrow M_1$ defined by $(x, y) \rightarrow x$ and $\pi_{2} : M_1 \times_f M_2 \rightarrow M_2$ defined by $(x, y) \rightarrow y$ are projections and further these maps become Riemannian submersion. Further, one can observe that the fibers $\{x\}\times M_2 = \pi_{1}^{-1}(x)$
	 and the leaves $M_1\times\{y\} = \pi_{2}^{-1}(y)$
 are Riemannian submanifolds of $M_1 \times_f M_2.$ Those vectors which are tangent to leaves
	and those tangent to fibers are called horizontal and vertical respectively. If
	$v \in T_xM_1, x \in M_1$ and $y \in M_2,$ then the lift $\bar{v}$ of $v$ to $(x, y)$ is the unique
	vector $T_{(x,y)}M$ such that $(\pi_{1*})\bar{v} = v.$ For a vector field $X \in\Gamma(T{M_1}),$ the
	lift of $X$ to $M=M_1\times_fM_2$ is the vector field $\bar{X}$ whose value at $(x, y)$ is the lift of $X_x$
	to $(x, y).$ Thus the lift of $X \in\Gamma(T{M_1})$ to $M_1\times_f M_2$ is the unique element of
$\Gamma(T(M_1\times M_2))$ that is $\pi_{1}$-related to $X.$ The set of all horizontal lifts is denoted
	$\mathcal{L}_H(M_1).$ Similarly, we denote the set of all vertical lifts by $\mathcal{L}_V(M_2).$ Thus a vector field $\bar{E}$ of $M_1 \times_fM_2$ can be expressed as
	$\bar{E} = \bar{X}+\bar{U}$ with $\bar{X}\in \mathcal{L}_H(M_1)$ and $\bar{U}\in \mathcal{L}_V(M_2).$ We can prove that
	$$\pi_{1*}(\mathcal{L}_H(M_1)) = \Gamma(TM_1) , \pi_{2*}(\mathcal{L}_V(M_2)) = \Gamma(TM_2)$$
	and so $\pi_{1*}(\bar{X}) = X \in\Gamma(TM_1)$ and $\pi_{2*}(\bar{U}) = U \in\Gamma(TM_2).$
	Throughout this paper, we take same notation for vector field and its lift to warped product manifold.
	\begin{definition}\cite{J14}
	Let $\phi_i:M_i\rightarrow N_i$ be the Riemannian maps from $M_i$ to $N_i,$ where $i=1, 2.$ Then the map
	$$\phi= \phi_1\times\phi_2:M=M_1\times_fM_2\rightarrow N = N_1\times_\rho N_2$$ defined as $(\phi_1\times\phi_2)(p_1, p_2) = (\phi_1(p_1), \phi_2(p_2))$ is a Riemannian warped product map.
	\end{definition}
	\begin{lemma}\label{lemma}\cite{J6}
		Let $M = M_1 \times_f M_2$ be a warped product manifold. For any $X_1, Y_1\in \mathcal{L}(M_1)$ and $X_2, Y_2 \in \mathcal{L}(M_2),$ we have
		\begin{enumerate}
			\item $\nabla^M_{X_1} Y_1$ is the lift of $\nabla^{M_1}_{X_1} Y_1,$
			\item  $\nabla ^M_{X_1} X_2 = \nabla^M_{X_2}X_1 = \frac{X_1(f)}{f}X_2,$
			\item  nor $(\nabla^M_{X_2}Y_2) = -g_M(X_2, Y_2)(\nabla^Mln f),$
			\item  tan $(\nabla^M_{X_2}Y_2)$ is the lift of $\nabla^{M_2}_{X_2}Y_2.$	
		\end{enumerate}
		Here $\nabla,\nabla^1$ and $\nabla^2$ denote Riemannian connections on $M, M_1$ and $M_2,$
		respectively.
	\end{lemma}
\begin{proposition}\cite{J15}
	Let $\phi_i: M_i\rightarrow N_i$ where i = 1, 2, be smooth functions. Then
	$$(\phi_1\times\phi_2)_*w = (\phi_{1*}u, \phi_{2*}v),$$
	where $w = (u, v) \in T_{(p_1, p_2)}(M_1 \times M_2).$
\end{proposition}
\begin{proposition}\label{Ismphm}\cite{J15}
	Let $\pi_{1}$ and $\pi_{2}$ be projections of $M_1\times_fM_2$ onto $M_1$ and $M_2$ respectively. Then $\lambda: T_{(p_1,p_2)}(M_1 \times M_2)\rightarrow T_{p_1}M_1\oplus T_{p_2}M_2,$ defined by $$\lambda(x) = (\pi_{1*}, \pi_{2*})x,$$ is an isomorphism.
\end{proposition}

	\section{Clairaut warped product Riemannian maps}
	\begin{theorem}\label{thm1}
		Let $\phi = \phi_1\times \phi_2: M = M_1\times_fM_2\rightarrow N = N_1\times_\rho N_2$ be a Riemannian warped product map between two warped product Riemannian manifolds. Let $\gamma$ be a regular curve on $M.$ Then we discuss the conditions of a curve $\gamma$ to be geodesic under following cases.
		\begin{enumerate}
			\item If $\dot{\gamma}$ is vertical vector field, i.e.,  $X_1$ and $X_2$ are vertical vector fields, then
				$$ T(U, U) = 0$$
				and
		  $$ \hat{\nabla}_{U}U = 0.$$
		  \item  If $\dot{\gamma}$ is horizontal vector field, i.e.,  both $X_1$ and $X_2$ are horizontal vector fields, then
		  $$A(Y, Y) = 0$$
		  and
		  $$\mathcal{H}(\nabla^M_Y Y)= 0.$$
		  \item If $\dot{\gamma}$ is neither vertical nor horizontal vector field, i.e.,  both $X_1$ and $X_2$ are neither horizontal nor vertical, then
		$$\hat{\nabla}^M_U U+T(U, Y)+\nabla^M_Y U+A_Y Y = 0$$
		and
		$$T(U, U)+\mathcal{H}\nabla^M_U Y+\mathcal{H}\nabla^M_Y Y+A(Y, U) = 0,$$
		\end{enumerate}
	where $ X_1\in\mathcal{L}(M_1)$ and $X_2\in\mathcal{L}(M_2).$
	\end{theorem}
\begin{proof}
	Let $\gamma : I\rightarrow M$ be a regular curve and $\dot{\gamma} = (X_1, X_2)$ where $X_1\in \mathcal{L}(M_1)$ and $X_2\in \mathcal{L}(M_2).$  Then 
	$$\nabla^M_{\dot{\gamma}}\dot{\gamma} = \nabla^M_{(X_1, X_2)}(X_1, X_2).$$
	Making use of Lemma \ref{Ismphm}, we can write
	$$\nabla^M_{\dot{\gamma}}\dot{\gamma} =  \nabla^M_{X_1}X_1+\nabla^M_{X_1}X_2+\nabla^M_{X_2}X_1+\nabla^M_{X_2}X_2.$$
	Using Proposition \ref{lemma}, we get
	\begin{equation}\label{dotgamma}
	\nabla^M_{\dot{\gamma}}\dot{\gamma} = \nabla^{M_1}_{X_1}X_1+2\frac{X_1(f)}{f}X_2-g_M(X_2, X_2)(\nabla^Mlnf) +\nabla^{M_2}_{X_2}X_2.
	\end{equation}
	We discuss above result under various cases.
	\begin{enumerate}
		\item If both $X_1$ and $X_2$ are vertical vector fields, i.e., $X_1 = U_1$ and $X_2 = U_2.$ Then from \eqref{dotgamma}, we have
		$$\nabla^M_{\dot{\gamma}}\dot{\gamma} = \nabla^{M_1}_{U_1}U_1+2\frac{U_1(f)}{f}U_2-g_M(U_2, U_2)(\nabla^Mlnf) +\nabla^{M_2}_{U_2}U_2.$$ 
		Using O'Neill tensor \eqref{Neill tensor} in above equation, we get
		\begin{equation}\label{Vdotgamma}
			\begin{split}
		\nabla^M_{\dot{\gamma}}\dot{\gamma} &= T_1(U_1, U_1)+\hat{\nabla}^1_{U_1}U_1+2\frac{U_1(f)}{f}U_2+T_2(U_2, U_2)+\hat{\nabla}^2_{U_2}U_2\\&-g_M(U_2, U_2)(\nabla^Mlnf). 		
			\end{split}
		\end{equation}
	
	Taking vertical component of \eqref{Vdotgamma}, we get
	 $$\mathcal{V}\nabla^M_{\dot{\gamma}}\dot{\gamma} = \hat{\nabla}^1_{U_1}U_1+2\frac{U_1(f)}{f}U_2+\hat{\nabla}^2_{U_2}U_2.$$
	Using Lemma 4 of \cite{J7} in above equation, we get
	$$\mathcal{V}\nabla^M_{\dot{\gamma}}\dot{\gamma} = \hat{\nabla}_U U.$$

		Taking horizontal component of \eqref{Vdotgamma}, we get
		$$\mathcal{H}\nabla^M_{\dot{\gamma}}\dot{\gamma} = T_1(U_1, U_1)+T_2(U_2, U_2)-g_M(U_2, U_2)(\nabla^Mlnf).$$
	Using Lemma 4 of \cite{J7} in above equation, we get
	$$\mathcal{H} \nabla^M_{\dot{\gamma}}\dot{\gamma} =T(U, U). $$
	 $\gamma$ is geodesic if and only if $\mathcal{V}	\nabla^M_{\dot{\gamma}}\dot{\gamma} = 0$ and $\mathcal{H}	\nabla^M_{\dot{\gamma}}\dot{\gamma} = 0.$
	
	\item  If both $X_1$ and $X_2$ are horizontal vector fields,
	i.e., $X_1 = Y_1$ and $X_2 = Y_2.$ Then from \eqref{dotgamma}, we obtain
	$$\nabla^M_{\dot{\gamma}}\dot{\gamma} = \nabla^{M_1}_{Y_1}Y_1+2\frac{Y_1(f)}{f}Y_2-g_M(Y_2, Y_2)(\nabla^Mlnf) +\nabla^{M_2}_{Y_2}Y_2.$$ 
	Using O'Neill tensor \eqref{Neill tensor} in above equation, we get
	\begin{equation}\label{Hdotgamma}
		\begin{split}
\nabla^M_{\dot{\gamma}}\dot{\gamma}& = A_1(Y_1, Y_1)	+ \mathcal{H}_1\nabla^{M_1}_{Y_1}Y_1+2\frac{Y_1(f)}{f}Y_2+A_2(Y_2, Y_2)+\mathcal{H}_2\nabla^{M_2}_{Y_2}Y_2\\&-g_M(Y_2, Y_2)(\nabla^M lnf).
		\end{split}
	\end{equation}
Taking vertical component of above equation, we get
$$\mathcal{V}\nabla^M_{\dot{\gamma}}\dot{\gamma} = A_1(Y_1, Y_1)+A_2(Y_2, Y_2).$$
Using Lemma 6 of \cite{J7} in above equation, we get
$$\mathcal{V}\nabla^M_{\dot{\gamma}}\dot{\gamma} = A(Y, Y)$$
and taking horizontal part of \eqref{Hdotgamma}, we get
$$\mathcal{H}\nabla^M_{\dot{\gamma}}\dot{\gamma} = \mathcal{H}_1\nabla^{M_1}_{Y_1}Y_1+2\frac{Y_1(f)}{f}Y_2+\mathcal{H}_2\nabla^{M_2}_{Y_2}Y_2-g_M(Y_2, Y_2)(\nabla^M lnf).$$
Using Lemma 6 of \cite{J7} in above equation, we get
$$\mathcal{H}\nabla^M_{\dot{\gamma}}\dot{\gamma} =\mathcal{H}(\nabla^M_Y Y).$$
	If $\gamma$ is geodesic if and only if $\mathcal{V}	\nabla^M_{\dot{\gamma}}\dot{\gamma} = 0$ and $\mathcal{H}	\nabla^M_{\dot{\gamma}}\dot{\gamma} = 0.$

\item If both $X_1$ and $X_2$ are neither vertical nor horizontal vector fields, i.e., $X_1 =  U_1+Y_1$ and $X_2 = U_2+Y_2,$ then
from \eqref{dotgamma}, we obtain 
\begin{equation*}
	\begin{split}
		\nabla^M_{\dot{\gamma}}\dot{\gamma}& = \nabla^{M_1}_{U_1+Y_1}(U_1+Y_1)+2\frac{(U_1+Y_1)(f)}{f}(U_2+Y_2)+\nabla^{M_2}_{U_2+Y_2}U_2+Y_2\\&-g_M(U_2+Y_2, U_2+Y_2)(\nabla^Mlnf). 		
	\end{split}
\end{equation*}
After simplification and using O'Neill tensor \eqref{Neill tensor}, we obtain
\begin{equation}\label{VHdotGamma}
	\begin{split}
		\nabla^M_{\dot{\gamma}}\dot{\gamma} &= T_1(U_1, U_1)+\hat{\nabla}^1_{U_1}U_1+T_1(U_1, Y_1)+\mathcal{H}_1\nabla^{M_1}_{U_1}Y_1+A_1(Y_1, U_1)+\\&\mathcal{V}_1\nabla^{M_1}_{Y_1}U_1+A_1(Y_1, Y_1)+\mathcal{H}_1\nabla^{M_1}_{Y_1}Y_1+2\frac{U_1(f)}{f}U_2+2\frac{U_1(f)}{f}Y_2\\&+2\frac{Y_1(f)}{f}U_2+2\frac{Y_1(f)}{f}Y_2-g_M(U_2, U_2)(\nabla^Mlnf)-g_M(Y_2, Y_2)(\nabla^Mlnf)\\&+T_2(U_2, U_2)+\hat{\nabla}^2_{U_2}U_2+T_2(U_2, Y_2)+\mathcal{H}_2\nabla^{M_2}_{U_2}Y_2+A_2(Y_2, U_2)\\&+\mathcal{V}_2\nabla^{M_2}_{Y_2}U_2+A_2(Y_2, Y_2)+\mathcal{H}_2\nabla^{M_2}_{Y_2}Y_2.		
	\end{split}
\end{equation}	 
Taking horizontal component of above equation, we have
\begin{equation}\label{Hcase3}
	\begin{split}
	\mathcal{H}\nabla^M_{\dot{\gamma}}\dot{\gamma}& = T_1(U_1, U_1)+T_2(U_2, U_2)+\mathcal{H}_1\nabla^{M_1}_{U_1}Y_1+\mathcal{H}_2\nabla^{M_2}_{U_2}Y_2+A_1(Y_1, U_1)\\&+A_2(Y_2, U_2)+\mathcal{H}_1\nabla^{M_1}_{Y_1}Y_1+\mathcal{H}_2\nabla^{M_2}_{Y_2}Y_2+2\frac{U_1(f)}{f}Y_2+2\frac{Y_1(f)}{f}Y_2\\&-g_M(U_2, U_2)(\nabla^M lnf)-g_M(Y_2, Y_2)(\nabla^M lnf).
	\end{split}
\end{equation}
With the help of Lemma 4, 5 and 6 of \cite{J7}, we get
\begin{equation*}
	\begin{split}
T_1(U_1, U_1)+T_2(U_2, U_2) &=T(U_1, U_1)+T(U_2, U_2)+g_M(U_2, U_2)(\nabla^M lnf)\\& = T(U_1+U_2, U_1+U_2)+g_M(U_2, U_2)(\nabla^M lnf)\\& = T(U, U)+g_M(U_2, U_2)(\nabla^M lnf),		
	\end{split}
\end{equation*}
\begin{equation*}
\begin{split}
\mathcal{H}_1\nabla^{M_1}_{U_1}Y_1+\mathcal{H}_2\nabla^{M_2}_{U_2}Y_2& = \mathcal{H}\nabla^{M}_{U_1}Y_1+\mathcal{H}\nabla^{M}_{U_2}Y_2\\& = \mathcal{H}\nabla^M_{(U_1+U_2)}(Y_1+Y_2)-\frac{U_1(f)}{f}Y_2\\& = \mathcal{H}\nabla^M_U Y-\frac{U_1(f)}{f}Y_2,
\end{split}
\end{equation*}

\begin{equation*}
	\begin{split}
\mathcal{H}_1\nabla^{M_1}_{Y_1}Y_1+\mathcal{H}_2\nabla^{M_2}_{Y_2}Y_2 &= \mathcal{H}\nabla^{M}_{Y_1}Y_1+\mathcal{H}\nabla^{M}_{Y_2}Y_2+g_M(Y_2, Y_2)(\nabla^M lnf)\\& = \mathcal{H}\nabla^{M}_{(Y_1+Y_2)}(Y_1+Y_2)-2\frac{Y_1(f)}{f}Y_2+g_M(Y_2, Y_2)(\nabla^M lnf)\\& = \mathcal{H}\nabla^M_Y Y -2\frac{Y_1(f)}{f}Y_2+g_M(Y_2, Y_2)(\nabla^M lnf),		
	\end{split}
\end{equation*}
and
\begin{equation*}
\begin{split}
A_1(Y_1, U_1)+A_2(Y_2, U_2)& = A(Y_1, U_1)+A(Y_2, U_2)\\& = A(Y, U)-\frac{U_1(f)}{f}Y_2.
\end{split}
\end{equation*}

Substituting all the above equations in \eqref{Hcase3}, we get
\begin{equation*}
\begin{split}
\mathcal{H}\nabla^M_{\dot{\gamma}}\dot{\gamma}& = T(U, U)+g_M(U_2, U_2)(\nabla^M lnf)+\mathcal{H}\nabla^M_U Y-\frac{U_1(f)}{f}Y_2+\mathcal{H}\nabla^M_Y Y\\& -2\frac{Y_1(f)}{f}Y_2+g_M(Y_2, Y_2)(\nabla^M lnf)+A(Y, U)-\frac{U_1(f)}{f}Y_2+2\frac{U_1(f)}{f}Y_2\\&-g_M(U_2, U_2)(\nabla^M lnf)-g_M(Y_2, Y_2)(\nabla^M lnf)+2\frac{Y_1(f)}{f}Y_2.
\end{split}
\end{equation*}
Therefore
\begin{equation*}
\mathcal{H}\nabla^M_{\dot{\gamma}}\dot{\gamma} = T(U, U)+H\nabla^M_U Y+\mathcal{H}\nabla^M_Y Y+A(Y, U),
\end{equation*}
and taking vertical component of \eqref{VHdotGamma}, we obtain
\begin{equation}\label{Vcase3}
\begin{split}
\mathcal{V}\nabla^M_{\dot{\gamma}}\dot{\gamma}& = \hat{\nabla}^1_{U_1}U_1+T_1(U_1, Y_1)+\mathcal{V}_1\nabla^{M_1}_{Y_1}U_1+A_1(Y_1, Y_1)+2\frac{U_1(f)}{f}U_2\\&+2\frac{Y_1(f)}{f}U_2+\hat{\nabla}^2_{U_2}U_2+T_2(U_2, Y_2)+\mathcal{V}_2\nabla^{M_2}_{Y_2}U_2+A_2(Y_2, Y_2).		
\end{split}
\end{equation}
With the help of Lemma 4, 5 and 6 of \cite{J7}, we get
$$\hat{\nabla}^1_{U_1}U_1+\hat{\nabla}^2_{U_2}U_2 = \hat{\nabla}_{U_1}U_1+\hat{\nabla}_{U_2}U_2 = \hat{\nabla}^M_U U-2\frac{U_1(f)}{f}U_2,$$
$$T_1(U_1, Y_1)+T_2(U_2, Y_2) = T(U_1, Y_1)+T(U_2, Y_2)= T(U, Y)-\frac{Y_1(f)}{f}U_2,$$
$$\mathcal{V}_1\nabla^{M_1}_{Y_1}U_1+\mathcal{V}_2\nabla^{M_2}_{Y_2}U_2 = \mathcal{V}\nabla^{M}_{Y_1}U_1+\mathcal{V}\nabla^M_{Y_2}U_2 = \nabla^M_Y U-\frac{Y_1(f)}{f}U_2,$$
and
$$A_1(Y_1, Y_1)+A_2(Y_2, Y_2) = A(Y_1, Y_1)+A(Y_2, Y_2) = A_Y Y.$$
Substituting all above equations
 in \eqref{Vcase3}, we get
\begin{equation}
\begin{split}
\mathcal{V}\nabla^M_{\dot{\gamma}}\dot{\gamma} = \hat{\nabla}^M_U U+T(U, Y)+\nabla^M_Y U+A_Y Y.
\end{split}
\end{equation}

If $\gamma$ is geodesic if and only if $\mathcal{V}\nabla^M_{\dot{\gamma}}\dot{\gamma} = 0$ and $\mathcal{H}\nabla^M_{\dot{\gamma}}\dot{\gamma} = 0.$\\
This completes the proof.

	\end{enumerate}
\end{proof}
	\begin{theorem}
Let $\phi = \phi_1\times\phi_2 : M = M_1\times_fM_2\rightarrow N = N_1\times_\rho N_2$ be a warped product Riemannian map between two warped product Riemannian manifolds. Then $\phi$ is a Clairaut Riemannian warped product map with $r = e^g$ if and only if $\phi_1$ has totally umbilical fibers and $\phi_2$ has totally geodesic fibers.
	\end{theorem}
\begin{proof}
Let $\gamma : I\rightarrow M$ be a geodesic on $M$ with $\mathcal{V}\dot{\gamma}(t) = U(t) = (U_1(t), U_2(t))$ and $\mathcal{H}\dot{\gamma}(t) = Y(t) = (Y_1(t), Y_2(t))$ and
let $\omega(t)$ denote the angle in $[0, \pi]$ between $\dot{\gamma}(t)$ and $Y(t).$ Assuming $b = ||\dot{\gamma}(t)||^2,$ we can write

\begin{equation}\label{cos sqr}
g_{\gamma(t)}(Y(t), Y(t)) = b cos^2\omega(t),
\end{equation}
\begin{equation}\label{sin sqr}
g_{\gamma(t)}(U(t), U(t)) = b sin^2\omega(t).
\end{equation}
Taking derivative of \eqref{cos sqr} w.r.t 't', we get
$$\frac{d}{dt}g_{\gamma(t)}(Y(t), Y(t)) = -2bcos\omega(t) sin\omega(t)\frac{d\omega}{dt}$$
which gives
$$g_M(\nabla^M_{\dot{\gamma}}Y, Y) = g_M(\nabla^M_{X_1+X_2}(Y_1+Y_2), Y_1+Y_2) = -bcos\omega(t)sin\omega(t)\frac{d\omega}{dt}.$$ 

%$$g_M(\nabla^M_{X_1+X_2}Y, Y) =  -bcos\omega(t)sin\omega(t)$$
%$$g_M(\nabla^M_{X_1+X_2}Y, Y) =  -bcos\omega(t)sin\omega(t)$$
%$$g_M(\nabla^M_{X_1}(Y_1+Y_2)+\nabla^M_{X_2}(Y_1+Y_2), Y_1+Y_2) =  -bcos\omega(t)sin\omega(t)$$
Using Lemma \ref{lemma} in above equation, we get 
\begin{equation*}
	\begin{split}
&g_M(\nabla^{M_1}_{X_1}Y_1+\frac{X_1(f)}{f}Y_2+\frac{Y_1(f)}{f}X_2+\nabla^{M_2}_{X_2}Y_2-g_M(X_2, Y_2)(\nabla^M lnf), Y)\\& = -bcos\omega(t)sin\omega(t)		
	\end{split}
\end{equation*}
which gives
\begin{equation}
\begin{split}
&g_M(\nabla^{M_1}_{U_1}Y_1+\nabla^{M_1}_{Y_1}Y_1+\frac{X_1(f)}{f}Y_2+\frac{Y_1(f)}{f}X_2+\nabla^{M_2}_{U_2}Y_2+\nabla^{M_2}_{Y_2}Y_2-g_M(X_2, Y_2)(\nabla^M lnf), Y)\\& = -bcos\omega(t)sin\omega(t).
\end{split}
\end{equation}

Using O'Neill tensor \eqref{Neill tensor} in above equation, we get
\begin{equation*}
	\begin{split}
	&g_M\big(T_1(U_1, Y_1)+\mathcal{H}_1\nabla^{M_1}_{U_1}Y_1+A_1(Y_1, Y_1)+\mathcal{H}_1\nabla^{M_1}_{Y_1}Y_1+\frac{X_1(f)}{f}Y_2+\frac{Y_1(f)}{f}X_2\\&+T_2(U_2, Y_2)+\mathcal{H}_2\nabla^{M_2}_{U_2}Y_2+A_2(Y_2, Y_2)+\mathcal{H}_2\nabla^{M_2}_{Y_2}Y_2-g_M(X_2, Y_2)(\nabla^M lnf), Y\big)\\& =  -bcos\omega(t)sin\omega(t)\frac{d\omega}{dt}
	\end{split}
\end{equation*}
which gives
\begin{equation*}
\begin{split}
&g_M(\mathcal{H}_1\nabla^{M_1}_{U_1}Y_1+\mathcal{H}_1\nabla^{M_1}_{Y_1}Y_1+\frac{X_1(f)}{f}Y_2+\frac{Y_1(f)}{f}X_2+\mathcal{H}_2\nabla^{M_2}_{U_2}Y_2+\mathcal{H}_2\nabla^{M_2}_{Y_2}Y_2-g_M(X_2, Y_2)(\nabla^M lnf), Y)\\& =  -bcos\omega(t)sin\omega(t)\frac{d\omega}{dt}.	
\end{split}
\end{equation*}

Using Lemma 5 and 6 of \cite{J7} in above equation, we get
%$$g_M(H\nabla^{M}_{U_1}Y_1+H\nabla^{M}_{Y_1}Y_1+\frac{X_1(f)}{f}Y_2+\frac{Y_1(f)}{f}X_2+H\nabla^{M}_{U_2}Y_2+H\nabla^{M}_{Y_2}Y_2) =  -bcos\omega(t)sin\omega(t)\frac{d\omega}{dt}$$
%Using , wet
$$g_M(\mathcal{H}\nabla^M_U Y+\mathcal{H}\nabla^M_Y Y, Y) = -bcos\omega(t)sin\omega(t)\frac{d\omega}{dt}.$$
Using Theorem \ref{thm1} in above equation, we get
\begin{equation}\label{AT}
g_M(T(U, U)+A(Y, U), Y) = bcos\omega(t)sin\omega(t)\frac{d\omega}{dt}.
\end{equation}
Since $\phi$ is a Clairaut warped product Riemannian map with $r = e^g$ if and only if $\frac{d}{dt}(e^{g\circ\gamma}sin\omega t) = 0$ this implies
$ e^{g\circ\gamma}cos\omega(t)\frac{d\omega}{dt}+e^{g\circ\gamma}sin\omega t\frac{d}{dt}(go\gamma) = 0.$\\
Multiplying by $bsin\omega(t)$, we get
$$g_M(U, U)g_M(\dot{\gamma}, \nabla^M g) = -bcos\omega(t)sin\omega(t)\frac{d\omega}{dt}.$$
From above equation and \eqref{AT}, we get
$$g_M(T(U, U), Y) = -g_M(U, U)g_M(\dot{\gamma}, \nabla^M g)$$
$$T(U, U) = -g_M(U, U)\nabla^M g$$
it gives that $\phi$ has totally umbilical fibers with mean curvature vector field $\vec{H}^{\phi} = -\nabla^M g.$ By Theorem 3 of \cite{J7}, we can conclude that $\phi$ is a Clairaut warped product Riemannian map if and only if $\phi_1$ has totally umbilical fibers and $\phi_2$ has totally geodesic fibers.

%Using prop. , we get
%$$g_M(\nabla^{M_1}_{X_1}Y_1 + \frac{X_1(f)}{f}Y_2+ \frac{Y_1(f)}{f}X_2-g_M(X_2, Y_2)(\nabla^M lnf) + \nabla^{M_2}_{X_2}Y_2,  Y_1+Y_2) = -bcos\omega(t)sin\omega(t)$$
%Using O"Neill tensor, we get\\
%$$g_M(\nabla^{M_1}_{U_1+Y_1}Y_1 + \frac{X_1(f)}{f}Y_2+ \frac{Y_1(f)}{f}X_2-g_M(X_2, Y_2)(\nabla^M lnf) + \nabla^{M_2}_{X_2}Y_2,  Y_1+Y_2) = -bcos\omega(t)sin\omega(t)$$
%$$g_M(H\nabla_{X_1}Y_1+\frac{X_1(f)}{f}Y_2+\frac{Y_1(f)}{f}X_2+H\nabla_{X_2}Y_2-g_M(X_2, Y_2)(\nabla^M lnf), Y) = -bcos\omega(t)sin\omega(t)$$
%Putting the value of above equation from
%$$g_M(T(U, U)+A(Y, U)-\frac{Y_1(f)}{f}X_2+\frac{X_1(f)}{f}Y_2, Y) = bcos\omega(t)sin\omega(t)$$
%Since $\phi$ is a Clairaut Riemannian map with $r = e^f$ if and only if $\frac{d}{dt}(e^{fo\gamma}sin\omega t) = 0$ this implies
%$ e^{fo\gamma}cos\omega(t)\frac{d\omega}{dt}+e^{fo\gamma}sin\omega t\frac{d}{dt}(fo\gamma) = 0.$
%Multiplying by $bcos\omega(t)$, we get
%$$g_M(U, U)g_M(\dot{\gamma}, \nabla f) = -bcos\omega(t)sin\omega(t)$$
%$$g_M(U, U)\nabla f = g_M(T(U, U)+A\big(Y, U)-\frac{Y_1(f)}{f}X_2+\frac{X_1(f)}{f}Y_2\big)$$

\end{proof}
%\begin{theorem}
%Let  $\phi = \phi_1\times\phi_2 : M = M_1\times_fM_2\rightarrow N = N_1\times_\rho N_2$ be a Clairaut warped product Riemannian map between two warped product Riemannian manifold. Then
%\end{theorem}
%\begin{proof}
%
%\end{proof}
\begin{theorem}
Let $\phi = (\phi_1, \phi_2) : M = M_1\times_f M_2\rightarrow N = N_1 \times_\rho N_2$ be a Clairaut Riemannian
warped product map between two Riemannian warped product
manifolds where $dimM_1 = m_1, dimM_2 = m_2, dimN_1 = n_1$ and $dimN_2 = n_2.$ Then Ricci curvature are defined as follows.
\begin{enumerate}
	\item $Ric(U_1, V_1) = \hat{Ric}^1(U_1, V_1)-(m_1-n_1)||\nabla^Mg||^2g_{M_1}(U_1, V_1)-g_{M_1}(U_1, V_1)div(\nabla^Mg)+g_{M_1}(A_{e_a}U_1, A_{e_a}V_1)-\frac{m_2}{f}H^f(U_1, V_1)$
	\item $Ric(U_2, V_2) = \hat{Ric}^2(U_2, V_2) +g_{M_2}(A_{\tilde{e_b}}U_2, A_{\tilde{e_b}}V_2)+(\frac{\Delta f}{f}-(m_2-1)\frac{||grad f||^2}{f^2})g_M(U_2, V_2)$
	
		\item $Ric(Y_1, Z_1) = Ric^{range\phi_{1*}}(\phi_{1*}Y_1, \phi_{1*}Z_1)-(m_1-n_1)g_{M_1}(\nabla^{M_1}_{Y_1}\nabla^Mg, Z_1)+g_{M_1}\big((\nabla_{e_i}A)_{Y_1}Z_1, e_i\big)-g_{M_1}(T_{e_i}Y_1, T_{e_i}Z_1)+g_{M_1}(A_{Y_1}e_i, A_{Z_1}e_i)+g_{N_1}\big((\nabla\phi_{1*})(Y_1, Z_1), \tau^{ker\phi_{1*}^\perp}\big)-g_{N_1}\big((\nabla\phi_{1*})(Y_1, e_a), (Z_1, e_a)\big)-\frac{m_2}{f}H^f(Y_1, Z_1)$
	\item $Ric(Y_2, Z_2) = Ric^{range\phi_{2*}}(\phi_{2*}Y_2, \phi_{2*}Z_2)+g_{N_2}\big((\nabla\phi_{2*})(Y_2, Z_2), \tau^{ker\phi_{2*}^\perp}\big)+g_{M_2}(A_{Y_2}\tilde{e_j}, A_{Z_2}\tilde{e_j})-g_{N_2}\big((\nabla\phi_{2*})(Y_2, \tilde{e_b}), (Z_2, \tilde{e_b})\big)+g_{M_2}\big((\nabla_{\tilde{e_j}}A)_{Y_2}Z_2, \tilde{e_j})+(\frac{\Delta f}{f}-(m_2-1)\frac{||grad f||^2}{f^2})g_M(Y_2, Z_2)$
	\end{enumerate}
	where $Y_i, Z_i\in\Gamma(\mathcal{H}_i), U_i, V_i\in\Gamma(\mathcal{V}_i), i=1,2.$
\end{theorem}
\begin{proof}
	Let us consider the orthonormal basis of $\mathcal{L}(M_1)$  in the following way
	$$\{e_1,..,e_{m_1-n_1},e_{m_1-n_1+1},...,e_{m_1}\},$$
	where $\{e_i:1\leq i\leq m_1-n_1\}$ and $\{e_a:m_1-n_1+1\leq a\leq m_1\}$ are the vertical and horizontal component of lift of vertical vector fields on $M_1.$ In similar way, we define the orthonormal basis of $\mathcal{L}(M_2)$ given by
	$$\{\tilde{e_1},...,\tilde{e}_{m_2-n_2}, \tilde{e}_{m_2-n_2+1},...,\tilde{e}_{m_2}\},$$
		where $\{\tilde{e_j}:1\leq j\leq m_2-n_2\}$ and $\{\tilde{e_b}:m_2-n_2+1\leq a\leq m_2\}$ are the vertical and horizontal component. With the help of Proposition 3 of \cite{J16} and Lemma 3.6 of \cite{J8}, we get the required result.

\end{proof}
\begin{theorem}
	Let $\phi = (\phi_1, \phi_2) : M = M_1\times_f M_2\rightarrow N = N_1 \times_\rho N_2$ be a Clairaut Riemannian
	warped product map between two Riemannian warped product
	manifolds. Then the sectional curvature is given by
	\begin{enumerate}
		\item $sec(U_1, V_1)= sec^1(U_1, V_1) = \hat{sec}^1(U_1, V_1)+|\nabla^Mg|^2$
		\item $sec(U_2, V_2) = sec^2(U_2, V_2) = \frac{\hat{sec}^2(U_2, V_2)-||\nabla f||^2}{f^2}$ 
		\item $sec(Y_1, Z_1) = \frac{g_{N_1}(R^{N_1}(\phi_{1*}Y_1, \phi_{1*}Z_1, \phi_{1*}Y_1, \phi_{1*}Z_1))-g_{N_1}((\nabla \phi_{1*})(Y_1, Y_1), (\nabla \phi_{1*})(Z_1, Z_1))+g_{N_1}((\nabla \phi_{1*})(Z_1, Y_1), (\nabla \phi_{1*})(Y_1, Z_1))}{|Y_1\wedge Z_1|^2}$
		\item $sec(U_1, Y_1) = \frac{g_{M_1}(U_1, U_1)g_{M_1}(\nabla_{Y_1}\nabla g, Y_1)+(Y_1(g))^2|U_1|^2-|A_{Y_1}U_1|^2}{|U_1|^2|Y_1|^2}$
		\item $sec(Y_2, Z_2) = \frac{g_{N_2}(R^{N_2}(\phi_{2*}Y_2, \phi_{2*}Z_2, \phi_{2*}Y_2, \phi_{2*}Z_2))-g_{N_2}((\nabla \phi_{2*})(Y_2, Y_2, ), (\nabla \phi_{2*})(Z_2, Z_2))+g_{N_2}((\nabla \phi_{2*})(Z_2, Y_2, ), (\nabla \phi_{2*})(Y_2, Z_2))-\frac{||\nabla f||^2}{f^2}}{|Y_2\wedge Z_2|^2}$
		\item $sec(U_2, Y_2) = -(\frac{|A_{Y_2}U_2|^2+||\nabla f||^2}{f^2|U_2|^2|Y_2|^2})$
	\end{enumerate}
\end{theorem}
\begin{proof}
With the help of Lemma 3.4 of \cite{J8} and from the paper \cite{J18}, we get the required result.
\end{proof}
%\begin{theorem}
%Let 	Let $\phi = (\phi_1, \phi_2) : M = M_1\times_{e^{\frac{-u}{m_1}}} M_2\rightarrow N = N_1 \times_\rho N_2$ be a Clairaut Riemannian
%warped product map where $M$ and $N$ are warped product manifold and Einstein manifold respectively admitting a non parallel vector field.
%\end{theorem}
%\begin{proof}
%
%
%\end{proof}

\section{Acknowledgement}

First author is grateful to the financial support provided by CSIR (Council
of science and industrial research) Delhi, India. File
no.[09/1051(12062)/2021-EMR-I]. The second author is thankful to the
Department of Science and Technology(DST) Government of India for providing
financial assistance in terms of FIST project(TPN-69301) vide the letter
with Ref No.:(SR/FST/MS-1/2021/104).\\

	\noindent J. Yadav and G. Shanker\newline
Department of Mathematics and Statistics\newline
Central University of Punjab\newline
Bathinda, Punjab-151401, India.\newline
Email: sultaniya1402@gmail.com; gauree.shanker@cup.edu.in\newline\\

\begin{thebibliography}{25}
	\bibitem{J16} Bishop RL., O'Neill, B., Manifolds of negative curvatur. Trans. Am. Math. Soc. 1969; 145, 1-49.
	\bibitem{J1} Bishop RL. Clairaut submersions. Differential geometry (in Honor of K-Yano), Kinokuniya, Tokyo, 1972; 21-31.
		\bibitem{J15}  Brickell F. and Clark RS., Differentiable Manifolds an Introduction, AVan Nostrand Reinhold Company Ltd., New York, 1970.
		\bibitem{J5} Chen BY. Warped product immersions, J. Geom.,  2005; 82, 36-49.
	\bibitem{J6} Chen BY. Differential Geometry of Warped Product Manifolds and Submanifolds, World Scientific Publishing Co. Pte. Ltd, Singapore, 2017.
		\bibitem{J7} Erken IK and Murathan C. Riemannian warped product submersions, Results
	Math., 2021; 76, 1-14.
	\bibitem{J8} Erken  IK., Murathan C. and  Siddiqui AN., Inequalities on Riemannian warped product submersions for Casorati curvatures, Mediterr. J. Math., 2023; 20(2), 1-18.
		\bibitem{J13} Falcitelli M., Ianus  S. and  Pastore AM., Riemannian Submersions and Related
	Topics, World Scientific, River Edge, NJ, 2004.
	 
	\bibitem{J2} Fischer AE. Riemannian maps between Riemannian manifolds. Contemporary Mathematics 1992; 132, 331-366.
	

	\bibitem{J14} Meena, K., Şahin, B., and Shah, H. M.  Riemannian Warped Product Maps. Results in Mathematics, 2024; 79(2), 56.
	
	\bibitem{J17} Moore, J.D. Isometric immersions of Riemannian products. J. Differ. Geom. 1971; 5, 159-168.

	\bibitem{J9} Nolker S., Isometric immersions of warped products, Differential Geom. Appl., 1996; 6,
	1-30.
	\bibitem{J10}  O'Neill B., The fundamental equations of a submersion, Michigan Math. J., 1966; 13(4),
	459-469.
	\bibitem{J11} O'Neill B., Semi-Riemannian Geometry with Applications to Relativity, Academic
	Press, New York, 1983.
		\bibitem{J3} \c{S}ahin B. Circles along a Riemannian map and Clairaut Riemannian maps. Bulletin of the Korean Mathematical Society 2017; 54 (1): 253-264
		\bibitem{J18} \c{S}ahin B. Notes on Riemannian maps. UPB Scientific Bulletin, Series A: Applied Mathematics and Physics 2017; 79 (3): 131-138.
	\bibitem{J12} Tojeiro R., Conformal immersions of warped products, Geom. Dedicata, 2007; 128,
	17-31.



	

\bibitem{J16} Yadav, A., and Meena, K.  Clairaut Riemannian maps whose total manifolds admit a Ricci soliton. International Journal of Geometric Methods in Modern Physics, 2022; 19(02), 2250024.
	\bibitem{J4} Yano K, Kon M. Structure on Manifolds. Singapore: World Scientific, 1984.
\end{thebibliography}
\end{document}